\documentclass[11pt]{article}
\usepackage[margin=1in]{geometry} 
\usepackage{amssymb,amsfonts,amsmath,bbm,mathrsfs,stmaryrd,mathtools}
\usepackage{xcolor}
\usepackage{url}

\usepackage{graphicx,dsfont} 

\usepackage{enumerate}
\usepackage[all]{xy}

\usepackage{hyperref}
\hypersetup{colorlinks,
             linkcolor=black!75!red,
             citecolor=blue,
             pdftitle={},
             pdfproducer={pdfLaTeX},
             pdfpagemode=None,
             bookmarksopen=true
             bookmarksnumbered=true}

\usepackage{tikz}
\usetikzlibrary{arrows,calc,decorations.pathreplacing,decorations.markings,intersections,shapes.geometric,shapes,through,fit,shapes.symbols,positioning,decorations.pathmorphing}

\usepackage{braket}

\usepackage[amsmath,thmmarks,hyperref]{ntheorem}
\usepackage{cleveref}

\creflabelformat{enumi}{#2(#1)#3}

\crefname{section}{Section}{Sections}
\crefformat{section}{#2Section~#1#3} 
\Crefformat{section}{#2Section~#1#3} 

\crefname{subsection}{\S}{\S\S}
\AtBeginDocument{%
  \crefformat{subsection}{#2\S#1#3}%
  \Crefformat{subsection}{#2\S#1#3}%
}

\crefname{subsubsection}{\S}{\S\S}
\AtBeginDocument{%
  \crefformat{subsubsection}{#2\S#1#3}%
  \Crefformat{subsubsection}{#2\S#1#3}%
}

\theoremstyle{plain}

\newtheorem{lemma}{Lemma}[section]
\newtheorem{proposition}[lemma]{Proposition}
\newtheorem{corollary}[lemma]{Corollary}
\newtheorem{theorem}[lemma]{Theorem}

\theoremstyle{nonumberplain}
\newtheorem{theoremN}{Theorem}
\newtheorem{propositionN}{Proposition}

\theoremstyle{plain}
\theorembodyfont{\upshape}
\theoremsymbol{\ensuremath{\blacklozenge}}

\newtheorem{definition}[lemma]{Definition}

\newtheorem{remark}[lemma]{Remark}
\newtheorem{convention}[lemma]{Convention}

\crefname{definition}{definition}{definitions}
\crefformat{definition}{#2definition~#1#3} 
\Crefformat{definition}{#2Definition~#1#3} 

\crefname{ex}{example}{examples}
\crefformat{example}{#2example~#1#3} 
\Crefformat{example}{#2Example~#1#3} 

\crefname{remark}{remark}{remarks}
\crefformat{remark}{#2remark~#1#3} 
\Crefformat{remark}{#2Remark~#1#3} 

\crefname{convention}{convention}{conventions}
\crefformat{convention}{#2convention~#1#3} 
\Crefformat{convention}{#2Convention~#1#3} 

\crefname{notation}{notation}{notations}
\crefformat{notation}{#2notation~#1#3} 
\Crefformat{notation}{#2Notation~#1#3} 

\crefname{table}{table}{tables}
\crefformat{table}{#2table~#1#3} 
\Crefformat{table}{#2Table~#1#3}

\crefname{lemma}{lemma}{lemmas}
\crefformat{lemma}{#2lemma~#1#3} 
\Crefformat{lemma}{#2Lemma~#1#3} 

\crefname{proposition}{proposition}{propositions}
\crefformat{proposition}{#2proposition~#1#3} 
\Crefformat{proposition}{#2Proposition~#1#3} 

\crefname{corollary}{corollary}{corollaries}
\crefformat{corollary}{#2corollary~#1#3} 
\Crefformat{corollary}{#2Corollary~#1#3} 

\crefname{theorem}{theorem}{theorems}
\crefformat{theorem}{#2theorem~#1#3} 
\Crefformat{theorem}{#2Theorem~#1#3} 

\crefname{enumi}{}{}
\crefformat{enumi}{(#2#1#3)}
\Crefformat{enumi}{(#2#1#3)}

\crefname{assumption}{assumption}{Assumptions}
\crefformat{assumption}{#2assumption~#1#3} 
\Crefformat{assumption}{#2Assumption~#1#3} 

\crefname{equation}{}{}
\crefformat{equation}{(#2#1#3)} 
\Crefformat{equation}{(#2#1#3)}

\numberwithin{equation}{section}
\renewcommand{\theequation}{\thesection-\arabic{equation}}

\theoremstyle{nonumberplain}
\theoremsymbol{\ensuremath{\blacksquare}}

\newtheorem{proof}{Proof}
\newcommand\pf[1]{\newtheorem{#1}{Proof of \Cref{#1}}}

\newcommand\bC{{\mathbb C}}

\newcommand\bG{{\mathbb G}}
\newcommand\bH{{\mathbb H}}

\newcommand\cH{{\mathcal H}}

\newcommand\cK{{\mathcal K}}

\newcommand\cV{{\mathcal V}}

\newcommand\fm{{\mathfrak m}}
\newcommand\fn{{\mathfrak n}}

\DeclareMathOperator{\id}{id}

\newcommand\numberthis{\addtocounter{equation}{1}\tag{\theequation}}

\global\long\def\tpr{\mathbin{\xymatrix{*+<.7ex>[o][F-]{\scriptstyle \bot}}
 } }%
\newcommand\li[1]{L^{\infty}(#1)}

\newcommand{\Ww}{\mathds{W}}
\newcommand{\wW}{\text{\reflectbox{$\Ww$}}\:\!}

\newcommand{\qedhere}{\mbox{}\hfill\ensuremath{\blacksquare}}

\title{Full quantum crossed products, invariant measures, and type-I lifting}
\author{Alexandru Chirvasitu}

\begin{document}

\date{}

\newcommand{\Addresses}{{
  \bigskip
  \footnotesize

  \textsc{Department of Mathematics, University at Buffalo, Buffalo,
    NY 14260-2900, USA}\par\nopagebreak \textit{E-mail address}:
  \texttt{achirvas@buffalo.edu}

}}

\maketitle

\begin{abstract}
  We show that for a closed embedding $\mathbb{H}\le \mathbb{G}$ of locally compact quantum groups (LCQGs) with $\mathbb{G}/\mathbb{H}$ admitting an invariant probability measure, a unitary $\mathbb{G}$-representation is type-I if its restriction to $\mathbb{H}$ is. On a related note, we also prove that if an action $\mathbb{G}\circlearrowright A$ of an LCQG on a unital $C^*$-algebra admits an invariant state then the full group algebra of $\mathbb{G}$ embeds into the resulting full crossed product (and into the multiplier algebra of that crossed product if the original algebra is not unital).
  
  We also prove a few other results on crossed products of LCQG actions, some of which seem to be folklore; among them are (a) the fact that two mutually dual quantum-group morphisms produce isomorphic full crossed products, and (b) the fact that full and reduced crossed products by dual-coamenable LCQGs are isomorphic.
\end{abstract}


\vspace{.5cm}



\section*{Introduction}

The paper fits within the framework of {\it locally compact quantum groups} (LCQGs, on occasion), in the sense of \cite{kvcast,kvwast,kus-univ} (with additional background in, say, \cite{wor-mult,sw1,sw2,mnw}). The results extend, to this quantum setting, a number of statements on (classical, i.e. non-quantum) locally compact groups that blend, in various proportions, the three items mentioned in the title. Relegating the background-recollection to \Cref{se:prel}, these are
\begin{itemize}
\item {\it full crossed products} $C_0^u(\widehat{\bG})\ltimes_f A $ attached to an action $\bG\circlearrowright A$ by an LCQG on a $C^*$-algebra $A$ (with $\widehat{\bG}$ denoting the Pontryagin dual of $\bG$).

  The usual classical construction (e.g. \cite[\S II.10.3.7]{blk}) transports to the quantum setting and features prominently, say, in \cite{vs-impr}.
\item {\it invariant measures}, which phrase, in the present non-commutative setting, refers to $\bG$-invariant states or functionals on $A$ acted upon by $\bG$.
\item {\it type-I lifting}, which refers to results of the following general shape:

  Given a closed embedding $\bH\le \bG$ with such and such a suite of properties (varying with the source/result), a unitary $\bG$ representation is of type I \cite[Definition 5.4.2]{dixc} provided its restriction to $\bH$ is. 
\end{itemize}

It is this latter strand of thought, in particular, that provided the initial motivation. Classically, such problems are posed and solved for instance in \cite{klm1,klm2,gtm,gk}. \cite[Conjecture II]{klm2}, for instance, states that type-I lifting in the above sense goes through if $\bH\le \bG$ is {\it cocompact}, i.e. the homogeneous space $\bG/\bH$ is compact.

Various cases of that conjecture are then proven, but one emerging pattern, of particular relevance to the present work, is that $\bG$-invariant measures on $\bG/\bH$ are extremely useful in delivering such lifting results. In fact, by \cite[Proposition 2.2]{klm2}, lifting holds as soon as such a finite measure exists, even {\it without} the compactness assumption on $\bG/\bH$.

This, then, brings two of the above-mentioned ingredients into the fold; the third (crossed products) comes about naturally as part of the standard induction-restriction toolkit for studying unitary representations. As is clear from a perusal of \cite{gk} and \cite[\S 1]{gtm} (modulo different notation), it is fruitful to
\begin{itemize}
\item recover $\bH$-representations as representations of the full crossed product $C_0^u(\widehat{\bG})\ltimes_f C_0(\bG/\bH)$ by {\it imprimitivity} \cite[\S 3.7]{mack-unit};  
\item and then make use of whatever additional structure the crossed product affords.
\end{itemize}

Much of this conceptual amalgam transports over to the quantum setting. We Summarize (and by necessity, abbreviate) some of the results.

The analogue of \cite[Proposition 2.2]{klm2} holds for quantum groups, as do some of the proof techniques (\Cref{th:exp} and \Cref{cor:slift,cor:wlift}):

\begin{theoremN}
  Let $\bH\le \bG$ be a closed quantum subgroup of an LCQG and assume $\bG/\bH$ admits a $\bG$-invariant finite measure. 
  \begin{enumerate}[(a)]
  \item For a unitary $\bG$-representation $U$, the commutant $R(\bH)'$ of the restriction $U|_{\bH}$ carries a normal conditional expectation onto the commutant $R(\bG)'$ of $U$.
  \item In particular, the latter commutant is type-I if the former is, and hence the type-I property for $U|_{\bH}$ entails that of $U$.
  \item Specializing once more, if $\bH$ is type-I then so is $\bG$. 
    \qedhere
  \end{enumerate}
\end{theoremN}

Classically, if $\bG/\bH$ carries a finite $\bG$-invariant measure {\it and} is compact, then, as a matter of presumably independent interest, the obvious map
\begin{equation*}
  C^*(\bG)\to C^*(\bG)\ltimes_f C(\bG/\bH)
\end{equation*}
is an embedding; this is noted in the course of the proof of \cite[Proposition 4.2]{gk} and gives an alternative proof for type-I lifting in this (more restrictive) case \cite[Corollary 4.5]{gk}. Once more, the quantum-flavored result holds; an amalgam of \Cref{th:prestate} and \Cref{cor:cpct}, somewhat weakened here for brevity reads

\begin{theoremN}
  If a $C^*$-algebra $A$ has an invariant state with respect to an action $\bG\circlearrowright A$, then the canonical $C^*$ morphism
  \begin{equation*}
    C_0^u(\widehat{\bG})\to M(C_0^u(\widehat{\bG})\ltimes_f A)
  \end{equation*}
  is an embedding, which factors through
  \begin{equation*}
    C_0^u(\widehat{\bG})\to C_0^u(\widehat{\bG})\ltimes_f A
  \end{equation*}
  if $A$ is in additional unital.  \qedhere
\end{theoremN}

Along the way, we also record a few scattered results on full LCQG crossed products for which it seems difficult to locate proofs in the literature, at least in this specific setup:

\begin{itemize}
  
\item One phenomenon that seems to be familiar is the ``reciprocity'' of \Cref{subse:recipr}: morphisms $\bH\to \bG$ and their duals $\widehat{\bG}\to \widehat{\bH}$ induce isomorphic full crossed products.

  This is what \cite[Proposition 2.5]{qg-full} boils down to, for instance, for the identity morphism of a classical locally compact group, while for closed embeddings of LCQGs an implicit application of the principle can be read into the chain of isomorphisms following \cite[Remark 6.5]{vs-impr}. \Cref{pr:sameprod} records the common generalization:
  \begin{propositionN}
    Consider an LCQG morphism $\bH\to \bG$, with its induced action $\bH\circlearrowright C_0^u(\bG)$, and the dual morphism $\widehat{\bG}\to \widehat{\bH}$ with {\it its} induced action $\widehat{\bG}\circlearrowright C_0^u(\widehat{\bH})$.

    We then have an isomorphism
    \begin{equation*}
      C_0^u(\widehat{\bH})\ltimes_f C_0^u(\bG)\cong C_0^u(\bG)\ltimes_f C_0^u(\widehat{\bH}).
    \end{equation*}
    between the two resulting full crossed products.  \qedhere
  \end{propositionN}

\item On the other hand, it is a well-established classical result that full and reduced crossed products by amenable locally compact groups are canonically isomorphic \cite[Theorem 7.7.7]{ped-aut}.

  References to a quantum version appear in the literature \cite[Remarques A.13 (c)]{bs-cross} and again the isomorphism chain on \cite[p.340, bottom]{vs-impr}. There is a proof in \cite[Proposition 5.6]{blnch} for regular multiplicative unitaries which likely extends. We give an alternative proof below that directly uses the existence of a counit on the reduced function algebra $C_0(\bG)$.

  The modern linguistic conventions require that we work with {\it dual-coamenable} LCQGs rather than amenable ones (see \Cref{def:amnb}), but with that in mind, \Cref{th:samecrossed} says precisely what is expected:
  
  \begin{theoremN}
    For an action $\bG\circlearrowright A$ of a dual-coamenable LCQG the canonical surjection
    \begin{equation*}
      C_0^u(\widehat{\bG})\ltimes_f A\to C_0(\widehat{\bG})\ltimes_r A
    \end{equation*}
    is an isomorphism.  \qedhere
  \end{theoremN}
\end{itemize}

\subsection*{Acknowledgements}

This work is partially supported by NSF grant DMS-2001128.

I am grateful for numerous enlightening conversations with J.Crann, M.Kalantar, P.Kasprzak, A.Skalski, P.So{\l}tan, S.Vaes and A.Viselter on much of the material below and related matters. 

\section{Preliminaries}\label{se:prel}

For preparatory material on operator algebras not recalled explicitly below the reader can consult any number of good sources, such as \cite{blk,ped-aut,tak1}, \cite{dixc} in conjunction with \cite{dixw}, etc.

Throughout the paper,
\begin{itemize}
\item $M(\cdot)$ denotes the {\it multiplier algebra} construction \cite[\S II.7.3]{blk} and as in \cite[discussion preceding D\'efinition 0.1]{bs-cross} we set, for a $C^*$-algebra $A$ and an ideal $J\trianglelefteq A$,
  \begin{equation}\label{eq:maj}
    M(A;\ J):=\{x\in M(A)\ |\ xA + Ax\subseteq J\}\subseteq M(J).
  \end{equation}
  The latter inclusion is by means of the restriction map
  \begin{equation*}
    M(A)\ni x\mapsto x|_{J}\in M(J),
  \end{equation*}
  which, as observed in loc.cit., {\it embeds} $M(A;\ J)$ into $M(J)$.
\item $K(\cH)$ and $B(\cH)\cong M(K(\cH))$ \cite[Theorem 15.2.12]{wo} are the algebras of bounded and compact operators on a Hilbert space $\cH$ respectively.
\item The tensor-product symbol `$\otimes$' denotes whatever version of that concept is appropriate, depending on the objects it is placed between:
  \begin{itemize}
  \item the usual tensor product of Hilbert spaces \cite[\S I.1.4.2]{blk};
  \item the {\it minimal} or {\it spatial} tensor product \cite[\S II.9.1.3]{blk} of $C^*$-algebras;
  \item similarly for von Neumann algebras (the spatial tensor product of \cite[\S III.1.5.4]{blk}).
  \end{itemize}
\end{itemize}

For $C^*$-algebras $A$ and $B$ the {\it strict} maps $A\to M(B)$ are of particular interest, and pervasive in the literature on locally compact quantum groups. Recall (e.g. \cite[Definition 2.3.1]{wo}) that the {\it strict topology} on $M(A)$ is that induced by the seminorms
\begin{equation*}
  x\mapsto \|ax\|,\ x\mapsto \|xa\|,\ a\in A
\end{equation*}
and (\cite[\S\S II.7.3.13 and II.7.5.1]{blk} or \cite[Notations and conventions]{kvcast}) that linear maps $A\to M(B)$ are strict if they are norm-bounded and continuous on the unit ball with respect to the strict topologies of $A\subset M(A)$ and $M(B)$.

\begin{remark}
  In the broader context of {\it Hilbert modules} the term `strict topology' can be ambiguous \cite[\S II.7.2.9]{blk}, but the ambiguities vanish for $M(A)$ \cite[II.7.3.1]{blk}. The upshot is that as far as $M(A)$ goes, all of the various notions of strictness coincide: \cite[discussion preceding Proposition 1.3]{lnc-hilb}, \cite[\S II.7.3.11]{blk}, etc.
\end{remark}

For {\it $C^*$-morphisms} $f:A\to M(B)$, in addition to strictness, {\it non-degeneracy} is another property of interest (\cite[\S II.7.3.8]{blk} or \cite[discussion preceding Proposition 2.1]{lnc-hilb}): the requirement that
\begin{equation*}
  f(A)B:=\mathrm{span}\{f(a)b\ |\ a\in A,\ b\in B\}
\end{equation*}
be (norm-)dense in $B$. Non-degeneracy implies strictness \cite[discussion preceding Proposition 5.5]{lnc-hilb}. In fact, the two properties can be characterized in terms of maps between the two multiplier algebras:
\begin{itemize}
\item strictness is equivalent to the extensibility of $f$ to a $C^*$ morphism $\overline{f}:M(A)\to M(B)$ strictly continuous on bounded sets;
\item while non-degeneracy is equivalent to said extensibility, plus the condition that the extension $\overline{f}$ be unital \cite[Proposition 2.5]{lnc-hilb}.
\end{itemize}

\begin{remark}
  To gain a fuller picture of the issue of the (unique) extensibility of maps $f:A\to M(B)$ to $M(A)$, recall further that
  \begin{itemize}
  \item {\it all} strict linear maps in the sense above have such an extension $\overline{f}:M(A)\to M(B)$, strictly continuous on bounded sets \cite[Proposition 7.2]{kus-1p};
  \item and that furthermore there is no distinction between bounded-set strict continuity and just plain strict continuity \cite[Corollary 2.7]{taylor}, so the qualification need not be observed.
  \end{itemize}
\end{remark}

\begin{convention}
  So pervasive (e.g. \Cref{def:equiv}) is the assumption of non-degeneracy for a $C^*$-algebra representation
  \begin{equation*}
    A\to B(\cH)\cong M(K(\cH))
  \end{equation*}
  (in the multiplier-algebra sense, or, equivalently, meaning \cite[\S II.6.1.5]{blk} that $A$ does not annihilate any non-zero vectors in $\cH$), that it will be profitable to simply assume non-degeneracy whenever representations are mentioned, unless specified otherwise.
\end{convention}

\subsection{Locally compact quantum groups}\label{subse:lcqg}

Much background on locally compact quantum groups is assumed implicitly, with \cite{kvcast,kvwast,kus-univ} serving as the main references and a few others mentioned explicitly below. For our purposes, the fastest entry point to locally compact quantum groups is probably \cite[Definition 1.1]{kvwast}.

\begin{definition}\label{def:lcqg}
  A {\it locally compact quantum group} (abbreviated {\it LCQG}) $\bG$ consists of
  \begin{itemize}
  \item a von Neumann algebra $M$, denoted by $\li{\bG}$, equipped with a normal, unital $*$-algebra morphism
    \begin{equation*}
      \Delta=\Delta_{\bG}:M\to M\otimes M,
    \end{equation*}
    {\it coassociative} in the sense that $(\id\otimes\Delta)\Delta=(\Delta\otimes\id)\Delta$.
  \item a normal, semifinite, faithful (nsf, for short) weight \cite[Definition VII.1.1]{tak2} $\varphi$ on $M$ (the {\it left Haar weight of $\bG$}), that is {\it left-invariant} in the sense that
    \begin{equation*}
      \varphi((\omega\otimes\id)\Delta(x)) = \omega(1)\varphi(x)
    \end{equation*}
    for all $\omega\in M_*^+$ and all
    \begin{equation*}
      x\in \fm_{\varphi}^+:=\{x\in M^+\ |\ \varphi(x)<\infty\}. 
    \end{equation*}
  \item similarly, an nsf weight $\psi$ (the {\it right Haar weight of $\bG$}), right-invariant in the sense that  \begin{equation*}
      \psi((\id\otimes\omega)\Delta(x)) = \omega(1)\psi(x)
    \end{equation*}
    for all $\omega\in M_*^+$ and $x\in \fm_{\psi}$.
  \end{itemize}
\end{definition}

Apart from \Cref{def:lcqg}, a few common items attached to an LCQG $\bG$ that will appear frequently are
\begin{itemize}
\item $C_0(\bG)\li{\bG}$, the {\it reduced function $C^*$-algebra} of $\bG$: this is the object introduced in \cite[Definition 4.1]{kvcast}, and is the focus of \cite{kvcast}.
\item the {\it universal} version $C_0^u(\bG)$ of the previous object: the $A_u$ of \cite[\S 5]{kus-univ} (where $A=C_0(\bG)$); it comes equipped with a surjection $C_0^u(\bG)\to C_0(\bG)$.
\item the GNS construction\cite[\S I.2]{strat}
  \begin{equation*}
    (L^2(\bG),\ \pi_{\varphi},\ \Lambda_{\varphi}) = (\cH_{\varphi},\ \pi_{\varphi},\ \Lambda_{\varphi})
  \end{equation*}
  attached to the left Haar weight of $\bG$.
\item the {\it (Pontryagin-)dual} LCQG $\widehat{\bG}$ constructed in \cite[\S 8]{kvcast}; $\li{\widehat{\bG}}$ is also naturally realized as a von Neumann subalgebra of $B(L^2(\bG))$, so this single Hilbert space carries the entire structure.
\item the {\it multiplicative unitary}
  \begin{equation*}
    W=W^{\bG}\in M(C_0(\bG)\otimes C_0(\widehat{\bG}))\subset B(L^2(\bG)\otimes L^2(\bG))
  \end{equation*}
  of \cite[p.873]{kvcast} (see also \cite[p.913, top]{kvcast} for the multiplier-algebra-membership claim).
\end{itemize}

For a {\it morphism} $\rho:\bH\to \bG$ of LCQGs (a notion studied extensively in its many guises in \cite{mrw}) we write
\begin{itemize}
\item
  \begin{equation*}
    \rho^u:C_0^u(\bG)\to M(C_0^u(\bH))
  \end{equation*}
  for the corresponding morphism of universal function $C^*$-algebras \cite[Theorem 4.8]{mrw};
\item
  \begin{equation*}
    \rho_l:L^{\infty}(\bG)\to L^{\infty}(\bH)\otimes L^{\infty}(\bG)
  \end{equation*}
  for the incarnation of $\rho$ as a left $\bH$-action on $\bG$ (using the same symbol for the other versions of this map, such as universal or reduced $C^*$ rather than $W^*$-algebras) \cite[Theorem 5.5]{mrw};
\item similarly,
  \begin{equation*}
    \rho_r:L^{\infty}(\bG)\to L^{\infty}(\bG)\otimes L^{\infty}(\bH)
  \end{equation*}
  for the right-handed version \cite[Theorem 5.3]{mrw}.
\item
  \begin{equation*}
    U^u_{\rho}\in M(C_0^u(\bH)\otimes C_0^u(\widehat{\bG}))
  \end{equation*}
  for the universal {\it bicharacter} associated to $\rho$ \cite[\S 3]{mrw} (noting the difference in handedness conventions).
\item $\widehat{\rho}:\widehat{\bG}\to \widehat{\bH}$ for its {\it dual} \cite[Corollary 4.3]{mrw}.
\end{itemize}

A {\it closed embedding} $\iota:\bH\le \bG$ (realizing $\bH$ as a {\it closed quantum subgroup} of $\bG$) is a morphism whose dual $\widehat{\iota}$ is given by an embedding
\begin{equation*}
  \li{\widehat{\bH}}\subseteq \li{\widehat{\bG}}
\end{equation*}
intertwining the comultiplications $\Delta_{\widehat{\bH}}$ and $\Delta_{\widehat{\bG}}$. These are (\cite[Theorem 3.3]{dkss}), in other words, the closed quantum subgroups of \cite[Definition 2.5]{vs-impr} and the {\it closed quantum subgroups in the sense of Vaes} of \cite[Definition 3.1]{dkss}.

An LCQG is $\bG$ {\it classical} if $C_0(\bG)$ is commutative, and hence the algebra of functions vanishing at infinity on an ordinary locally compact group, and {\it dual-classical} if $\widehat{\bG}$ is classical. Other LCQG-specific notions (unitary representations, actions, crossed products, etc.) will be recalled below, as needed.

\section{Remarks on full crossed products}\label{se:fc} 

Crossed products by quantum-group actions are studied extensively in \cite{bs-cross} (see also \cite[\S 2.3]{vs-impr} for a brief refresher). We recall some of the concepts. First, following \cite[D\'efinition 0.3]{bs-cross}, \cite[Definition 2.3]{vs-impr}, \cite[Definition 2.1]{dsv}, etc.:

\begin{definition}\label{def:urep}
  A {\it unitary representation} of an LCQG $\bG$ on a Hilbert space $\cK$ is a unitary $U\in M(C_0(\bG)\otimes K(\cK))$ such that $(\Delta\otimes \id)U = U_{12}U_{13}$.

  Equivalently, it is enough to require that $U\in \li{\bG}\otimes B(\cK)$; see \cite[discussion following Definition 2.1]{dsv}, which in turn cites \cite[Theorem 4.12]{bds}.

  A unitary representation $U$ as above can also be recast as a non-degenerate $C^*$ morphism $\pi_U:C_0^u(\widehat{\bG})\to B(\cK)$; $U$ and $\pi_U$ determine each other uniquely \cite[Proposition 5.2]{kus-univ}: the bijective correspondence $U\leftrightarrow \pi_U$ is given by
  \begin{equation}\label{eq:upiu}
    U = (\id\otimes \pi_U)\wW,
  \end{equation}
  where $\wW\in M(C_0(\bG)\otimes C_0^u(\widehat{\bG}))$ is the right-half-universal multiplicative unitary denoted by $\widehat{\cV}$ in \cite[Proposition 4.2]{kus-univ}. 
\end{definition}

{\it Actions} of an LCQG on a $C^*$-algebra $A$, with a caveat (cf. \cite[D\'efinition 0.2]{bs-cross} and \Cref{re:vsact}), are the {\it continuous coactions} of \cite[Definition 2.6]{vs-impr}.

\begin{definition}\label{def:act}
  Let $\bG$ be an LCQG, $A$ a $C^*$-algebra, and $\widetilde{A}$ its {\it unitization} \cite[\S II.1.2]{blk}.

  An {\it action} of $\bG$ on $A$ is a non-degenerate $C^*$ morphism
  \begin{equation}\label{eq:act}
    \rho:A\to M(C_0(\bG)\otimes A)
  \end{equation}
  such that
  \begin{enumerate}[(a)]
  \item\label{item:12} $(\id\otimes\rho)\rho = (\Delta_{\bG}\otimes\id)\rho$;
  \item\label{item:13} $\rho$ takes values in the subalgebra
    \begin{equation*}
      M\left(C_0(\bG)\otimes \widetilde{A};\ C_0(\bG)\otimes A\right)\subseteq M(C_0(\bG)\otimes A)
    \end{equation*}
    defined as in \Cref{eq:maj}, so that
    \begin{equation*}
      \rho(A)(C_0(\bG)\otimes 1)\subseteq C_0(\bG)\otimes A;
    \end{equation*}
  \item\label{item:14} and $\rho(A)(C_0(\bG)\otimes 1)$ is dense in $C_0(\bG)\otimes A$.
  \end{enumerate}
  We occasionally depict actions as circular arrows, as in the Introduction: $\bG\circlearrowright A$ or $\rho:\bG\circlearrowright A$. 
\end{definition}

The caveat alluded to before is condition \Cref{item:13}; this is presumably intended in \cite[Definition 2.6]{vs-impr} (though not mentioned explicitly), since otherwise there is no reason, a priori, why $\alpha(B)(A\otimes 1)$ would be contained in $A\otimes B$.

As for unitary $\bG$-representations compatible with $\bG$-actions on $C^*$-algebras (see \cite[1.2. Exemples. (5)]{bs-cross} or \cite[p.325]{vs-impr}, where such representations are termed `covariant'):

\begin{definition}\label{def:equiv}
  For an action $\rho:A\to M(C_0(\bG)\otimes A)$ of an LCQG $\bG$ on a $C^*$-algebra $A$ a {\it $\rho$-equivariant (or $\bG$-equivariant) representation} on a Hilbert space $\cK$ is a pair $(U,\pi)$ where
  \begin{itemize}
  \item $U\in M(C_0(\bG)\otimes K(\cK))$ is a unitary $\bG$-representation;
  \item $\pi:A\to B(\cK)$ is a non-degenerate representation of the $C^*$-algebra on the same Hilbert space;
  \item and the {\it equivariance} condition
    \begin{equation}\label{eq:defequivariance}
      (\id\otimes\pi)\rho(a) = U^*(1\otimes \pi(a))U,\ \forall a\in A.
    \end{equation}
  \end{itemize}
  We regard `covariant' and `equivariant' (and similarly, `covariance' and `equivariance') as synonymous, in order to preserve agreement with the cited sources.

  We will also occasionally package the Hilbert space into the mix, in order to display the symbol denoting it; thus, an equivariant representation $(U,\pi)$ on $\cK$ might be depicted as $(U,\pi,\cK)$.
\end{definition}

One can now mimic the classical procedure (e.g. \cite[\S\S II.10.3.7, II.10.3.14]{blk}) of defining both {\it full} and {\it reduced crossed products}: see for instance
\begin{itemize}
\item \cite[\S 2.3]{vs-impr} (where the full version of the crossed product is defined implicitly, via its universality property);
\item \cite[D\'efinition 7.1]{bs-cross} for reduced crossed products and \cite[Remarques A.13 (b)]{bs-cross} for a mention that covariant representations can be used to define full crossed products;
\item \cite[D\'efinition 5.3]{blnch} in the context of regular multiplicative unitaries;  
\item which is then extended in \cite[D\'efinition 4.2]{vrgn-phd} in sufficient generality (the {\it weak Kac systems} of \cite[p.39]{vrgn-phd}).
\end{itemize}

\Cref{def:fullcr} below retraces \cite[Definition 2.3]{qg-full}. 

\begin{definition}\label{def:fullcr}
  Let $\rho:A\to M(C_0(\bG)\otimes A)$ be an action of an LCQG on a $C^*$-algebra.
  \begin{enumerate}[(a)]
  \item\label{item:8} For a $\rho$-equivariant representation $(U,\pi)$ on $\cK$ we define the $C^*$-algebra
    \begin{align*}
      C^*(U,\pi) &:= \left(\text{$C^*$-algebra generated by }\pi(A)\cdot \pi_U(C_0^u(\widehat{\bG}))\right)
      \numberthis\label{eq:cupi}\\
                 &=\left(\text{$C^*$-algebra generated by }\pi_U(C_0^u(\widehat{\bG}))\cdot \pi(A)\right)
                   \subseteq B(\cK),
    \end{align*}
    where $\pi_U:C_0^u(\widehat{\bG})\to B(\cK)$ is as in \Cref{def:urep}.
  \item\label{item:9} An equivariant representation $(U,\pi)$ {\it weakly contains} another, $(U',\pi',\cK')$, if there is a representation
    \begin{equation*}
      \psi:C^*(U,\pi)\to B(\cK')\text{ with }\psi\circ \pi_U = \pi_{U'}\text{ and }\psi\circ \pi = \pi'.
    \end{equation*}
  \item\label{item:10} A $\rho$-equivariant representation is {\it weakly universal} if it weakly contains every other equivariant representation.
  \item\label{item:11} We define the {\it full crossed product}
    \begin{equation*}
      C_0^u(\widehat{\bG})\ltimes_f A := C^*(U,\pi)
    \end{equation*}
    for a weakly universal covariant representation $(U,\pi)$. 
  \end{enumerate}
\end{definition}

For {\it reduced} crossed products, one can adopt either of two approaches. 

\begin{definition}\label{def:redcr}
  Let $\rho:A\to M(C_0(\bG)\otimes A)$ be an action of an LCQG on a $C^*$-algebra.
  \begin{itemize}
  \item First, per \cite[D\'efinition 7.1]{bs-cross} or \cite[p.324]{vs-impr}: the {\it reduced crossed product} attached to $\rho$ is
  \begin{align*}
    C_0(\widehat{\bG})\ltimes_r A &:= \left(\text{$C^*$-algebra generated by }\rho(A)\cdot (C_0(\widehat{\bG})\otimes \bC)\right)                                    
                                    \subseteq M(K(L^2(\bG))\otimes A). 
  \end{align*}
\item Equivalently, by analogy to the classical (or dual-classical) constructions in, say, \cite[\S\S II.10.3.14, II.3.10.14]{blk} or \cite[Definition 2.7]{qg-full}:
  \begin{enumerate}[(a)]
  \item For any representation $\sigma:A\to B(\cK)$ consider the $\rho$-equivariant representation
    \begin{equation*}
      (U_{\sigma},\ \widetilde{\sigma})
      :=
      (W\otimes 1,\ (\pi_{\bG}\otimes \sigma)\rho)
    \end{equation*}
    on $L^2(\bG)\otimes \cK$, where $W\in M(C_0(\bG)\otimes C_0(\widehat{\bG}))$ is the multiplicative unitary of $\bG$ and $\pi_{\bG}:C_0(\bG)\to B(L^2(\bG))$ is the GNS representation of the left Haar weight of $\bG$.
  \item The {\it reduced crossed product} attached to $\rho$ is
    \begin{equation*}
      C_0(\widehat{\bG})\ltimes_r A
      :=
      C^*(U_{\sigma},\ \widetilde{\sigma})
    \end{equation*}
    as in \Cref{def:fullcr} \Cref{item:8} for a {\it faithful} representation $\sigma:A\to B(\cK)$.
  \end{enumerate}
  \end{itemize}
  It is not difficult to see that in the second definition the resulting $C^*$-algebra does not depend on $\sigma$ (so long as the latter is faithful) and that the two definitions are indeed equivalent.
\end{definition}

Given an action \Cref{eq:act} and an equivariant representation $(U,\pi,\cK)$ one expects the corresponding algebra $C^*(U,\pi)$ of \Cref{eq:cupi} to come equipped with non-degenerate morphisms
\begin{equation}\label{eq:wantnd}
  C_0^u(\bG) \to M(C^*(U,\pi))\quad\text{and}\quad A \to M(C^*(U,\pi)).
\end{equation}

This is indeed the case according to \cite[Lemme 4.1]{vrgn-phd}, whose first two items we paraphrase as follows.

\begin{proposition}\label{pr:havendeg}
  Let $(U,\pi,\cK)$ be a $\rho$-equivariant representation for an action $\rho:\bG\circlearrowright A$.
  \begin{enumerate}[(1)]
  \item The $C^*$-algebra $C^*(U,\pi)$ of \Cref{eq:cupi} can be recovered as
    \begin{align*}
      C^*(U,\pi) &= \overline{\pi_U(C_0^u(\widehat{\bG}))\cdot \pi(A)}^{\|\cdot\|}\\
                 &= \overline{\pi(A)\cdot \pi_U(C_0^u(\widehat{\bG}))}^{\|\cdot\|},
    \end{align*}
    i.e. these linear spans are already closed under multiplication and `$*$'. 
  \item The non-degenerate morphisms
    \begin{equation*}
      \pi_U:C_0^u(\bG) \to B(\cK)\quad\text{and}\quad \pi:A \to B(\cK)
    \end{equation*}
    factor through non-degenerate morphisms \Cref{eq:wantnd}.  \qedhere
  \end{enumerate}  
\end{proposition}

\begin{remark}
  As noted in \cite[\S 2.3]{vrgn-phd}, that paper's weak-Kac-system formalism applies to locally compact quantum groups formalized as in \cite{kvcast}; we will thus apply results in \cite{vrgn-phd} freely, as needed.
\end{remark}

Other, earlier versions of the result appear as \cite[Lemma 2.5]{lprs}, \cite[Lemme 7.2 and following discussion]{bs-cross} and \cite[Proposition 5.2]{blnch}; the proofs all ultimately revolve around the same idea of factoring an arbitrary functional $\omega\in B(L^2(\bG))$ as $\omega' x$ for some
\begin{equation*}
  \omega'\in B(L^2(\bG))_*\quad\text{and}\quad x\in C_0(\bG);
\end{equation*}
this is a consequence, say, of the {\it Cohen factorization theorem} \cite[p.750, Theorem]{lprs} or \cite[Theorem 32.22]{hr2}.

\begin{remark}\label{re:vsact}
  In particular, applying \Cref{pr:havendeg} to a weakly universal equivariant representation in the sense of \Cref{def:fullcr} \Cref{item:10} we obtain the universal unitary
  \begin{equation*}
    U_u\in M(C_0(\bG)\otimes C_0^u(\widehat{\bG})\ltimes_f A)
  \end{equation*}
  and non-degenerate representation
  \begin{equation*}
    \pi_u:A\to M(C_0^u(\widehat{\bG})\ltimes_f A)
  \end{equation*}
  denoted on \cite[p.325]{vs-impr} (where $U_u$ is denoted by $X_u$).

  Note, though, that the proofs implying the existence of these two objects rely crucially on actions \Cref{eq:act} actually taking values in $M\left(C_0(\bG)\otimes \widetilde{A};\ C_0(\bG)\otimes A\right)$. The same constraint is imposed in \cite[Definition 2.1]{qg-full} and \cite[Definition 2.1]{lprs} in the context of dual-classical LCQGs, both sources using the notation
  \begin{equation*}
    \widetilde{M}(A\otimes B) := M(\widetilde{A}\otimes B;\ A\otimes B),
  \end{equation*}
  with the latter identified with a $C^*$-subalgebra of $M(A\otimes B)$. \cite[\S 1]{lprs}, in particular, contains an illuminating discussion of the advantages of $\widetilde{M}$ over $M$.
\end{remark}

\subsection{Full-crossed-product reciprocity}\label{subse:recipr}

The following remark, for closed embeddings, is implicit in the computation on \cite[p.340, bottom]{vs-impr}. 

\begin{proposition}\label{pr:sameprod}
  Let $\rho:\bH\to \bG$ be an LCQG morphism. We have the following isomorphism between the universal crossed products built out of the left actions induced by $\rho$ and its dual $\widehat{\rho}:\widehat{\bG}\to \widehat{\bH}$:
  \begin{equation*}
    C_0^u(\widehat{\bH})\ltimes_f C_0^u(\bG)\cong C_0^u(\bG)\ltimes_f C_0^u(\widehat{\bH}).
  \end{equation*}
\end{proposition}
\begin{proof}
  Denote by $B^u$ and $A^u$ the universal $C^*$-algebras attached to $\bH$ and $\bG$ respectively, and similarly for duals (with $\widehat{\bullet}$ indicating duality); the target isomorphism then becomes
  \begin{equation*}
    \widehat{B}^u\ltimes_f A^u\cong A^u\ltimes_f \widehat{B}^u.
  \end{equation*}
  According to \Cref{def:fullcr}, a $\widehat{B}^u\ltimes_f A^u$-representation on a Hilbert space $\cH$ consists of
  \begin{itemize}
  \item A representation $\pi:A^u\to B(\cH)$ of the $C^*$-algebra $A^u$ on $H$;
  \item and a unitary representation
    \begin{equation}\label{eq:x}
      X\in M(B^u\otimes K(\cH))
    \end{equation}
    of the quantum group $B$;
  \item such that
    \begin{equation}\label{eq:cov}
      (\id\otimes \pi)\rho_l(x) = X^*(1\otimes \pi(x))X\in ,\ \forall x\in A^u. 
    \end{equation}
  \end{itemize}
  We write $X_l$ for $X$, to highlight the left-hand placement of the $A$-leg in \Cref{eq:x} (there will be an $X_r$ shortly). 

  Now, the universality property of $A^u$ provides \cite[Proposition 5.3 and/or dual to Proposition 6.5]{kus-univ} a bijection between $\pi:A^u\to B(\cH)$ as above and unitaries
  \begin{equation*}
    X_r\in M(K(\cH)\otimes \widehat{A}^u)
  \end{equation*}
  (satisfying appropriate conditions).

  Consider the universal unitary $U^u=U^u_A\in M(A^u\otimes \widehat{A}^u)$ of \cite[Proposition 6.4]{kus-univ}. Applying the two sides of \Cref{eq:cov} to its left leg and using the identity
  \begin{equation*}
    (\Delta\otimes\id)U = U_{13}U_{23},
  \end{equation*}
  \Cref{eq:cov} translates to 
  \begin{equation}\label{eq:pentagon}
    U_{\rho 13}X_{r23} = X^*_{l12}X_{r23} X_{l12},    
  \end{equation}
  where
  \begin{equation*}
    U^u_{\rho} = (\rho^u\otimes\id)U^u_A\in M(B^u\otimes \widehat{A}^u)
  \end{equation*}
  is the bicharacter associated to $\rho$. Moving two factors around, this is equivalent to
  \begin{equation}\label{eq:movedpent}
    X_{l12}U^u_{\rho 13} = X_{r23} X_{l12} X^*_{r23}.
  \end{equation}
  Now,
  \begin{itemize}
  \item applying the $*$ operation to \Cref{eq:movedpent};
  \item and also reversing the tensorands
  \end{itemize}
  transforms that equation into
  \begin{equation}\label{eq:dualpent}
    \widehat{U^u_{\rho}}_{13}\widehat{X}_{r23} = X^*_{l12}\widehat{X}_{r23} \widehat{X}_{l12},
  \end{equation}
  where
  \begin{itemize}
  \item
    \begin{equation*}
      \widehat{U^u_{\rho}} := \text{flip}(U^{u*}_{\rho}) = U^u_{\widehat{\rho}}\in M(\widehat{A}^u\otimes B^u)
    \end{equation*}
    is the universal bicharacter attached to the dual quantum-group morphism $\widehat{\rho}:\widehat{\bG}\to \widehat{\bH}$ \cite[Proposition 3.9]{mrw};
  \item and similarly,
    \begin{equation*}
      \widehat{X}_l := \text{flip}(X^*_r)\in M(\widehat{A}^u\otimes K(\cH));
    \end{equation*}
  \item and
    \begin{equation*}
      \widehat{X}_r := \text{flip}(X^*_l)\in M(K(\cH)\otimes A^u).
    \end{equation*}
  \end{itemize}
  Note, though, that \Cref{eq:dualpent} is precisely \Cref{eq:pentagon} decorated with hats; in other words, specifying a pair $(X_l,X_r)$ is the same as specifying a pair $(\widehat{X}_l,\widehat{X}_r)$, playing the same role for the dual group morphism $\widehat{\rho}$. This means that the representations of $\widehat{B}^u\ltimes_f A^u$ and $A^u\ltimes_f \widehat{B}^u$ are classified by the same data, and we are done.
\end{proof}

\begin{remark}
  \cite[Proposition 2.5]{qg-full} is, essentially, \Cref{pr:sameprod} applied to the identity morphism on a classical locally compact group.
\end{remark}

\subsection{Coamenability and full crossed products}\label{subse:coam}

\Cref{th:samecrossed} below is, presumably, well known:
\begin{itemize}
\item it is a straightforward quantum generalization of the fact that full and reduced crossed products by amenable locally compact groups coincide \cite[Theorem 7.7.7]{ped-aut};
\item it is invoked implicitly in \cite[paragraph following Remark 6.5]{vs-impr};
\item claimed explicitly in \cite[Remarques A.13 (c)]{bs-cross} in the setting of regular LCQGs;
\item proven in \cite[Proposition 5.6]{blnch} for {\it regular} multiplicative unitaries (and hence locally compact quantum groups);
\item and presumably the proof extends generally, as the statement \cite[D\'efinition]{vrgn-phd} seems to imply.
\end{itemize}

The earlier sources refer to {\it amenable} locally compact quantum groups (`moyennable' in \cite[Remarques A.13 (c)]{bs-cross}). To preserve agreement with the language of \cite{bt} (now in wide use), the term adopted below would is `dual-coamenable', i.e. an LCQG $\bG$ whose dual $\widehat{\bG}$ is coamenable in the following sense (\cite[Definition 3.1, Theorem 3.1]{bt}).

\begin{definition}\label{def:amnb}
  An LCQG $\bG$ is {\it coamenable} if either of the two following equivalent conditions holds
  \begin{itemize}
  \item The surjection $C_0^u(\bG)\to C_0(\bG)$ is an isomorphism.
  \item The reduced function algebra $C_0(\bG)$ has a {\it counit}: a $C^*$-morphism $\varepsilon:C_0(\bG)\to \bC$ such that $(\id\otimes\varepsilon)\Delta_{\bG}=\id$.
  \end{itemize}

  $\bG$ is {\it dual-coamenable} if $\widehat{\bG}$ is coamenable.
\end{definition}

The proof of \Cref{th:samecrossed} below proceeds along the same lines as that of the dual-classical \cite[Theorem 3.7]{lprs}, directly using the counit (rather than alternative characterizations of coamenability, as in \cite[Proposition 5.6]{blnch}).

\begin{theorem}\label{th:samecrossed}
  For an action $\rho:A\to M(C_0(\bG)\otimes A)$ of a dual-coamenable locally compact quantum group the canonical morphism
  \begin{equation*}
    C_0^u(\widehat{\bG})\ltimes_f A\to C_0(\widehat{\bG})\ltimes_r A
  \end{equation*}
  is an isomorphism.
\end{theorem}
\begin{proof}
  We have to show that every pair consisting of
  \begin{itemize}
  \item a representation
    \begin{equation}\label{eq:pi}
      \pi:A\to B(\cH);
    \end{equation}
  \item and a $\bG$-representation
    \begin{equation}\label{eq:X}
      X\in M(C_0(\bG)\otimes K(\cH));
    \end{equation}
  \item satisfying the equivariance condition
    \begin{equation}\label{eq:covagain}
      (\id\otimes \pi)\rho(a) = X^*(1\otimes \pi(a))X\in ,\ \forall a\in A
    \end{equation}
  \end{itemize}
  arises from a representation of the {\it reduced} crossed product $C_0(\widehat{\bG})\ltimes_r A$ on the same Hilbert space $\cH$.

  Consider the representation
  \begin{equation*}
    M(K(L^2(\bG))\otimes A)\ni x\otimes a\stackrel{\theta}{\longmapsto} X(x\otimes \pi(a))X^* \in B(L^2(\bG)\otimes \cH),
  \end{equation*}
  and restrict it to the reduced crossed product (\cite[\S 2.3]{vs-impr})
  \begin{equation}\label{eq:redcp}
    C_0(\widehat{\bG})\ltimes_r A = \overline{\mathrm{span}}\{\rho(a)(x\otimes 1)\ |\ x\in C_0(\widehat{\bG}),\ a\in A\}\subset M(K(L^2(\bG))\otimes A).
  \end{equation}
  Next, apply $\theta$ to the typical element $\rho(a)(x\otimes 1)$ displayed in \Cref{eq:redcp}, obtaining
  \begin{equation}\label{eq:applytheta}
    X(\id\otimes\pi)\rho(a) X^* X(x\otimes 1)X^* = (1\otimes\pi(a))X(x\otimes 1)X^*\in B(L^2(\bG)\otimes \cH),
  \end{equation}
  via the equivariance condition \Cref{eq:covagain}. 

  Denote by $\pi_X:C_0^u(\widehat{\bG})\to B(\cH)$ attached to $X$ and by
  \begin{equation*}
    U\in M(C_0(\bG)\otimes C_0^u(\widehat{\bG}))
  \end{equation*}
  the ``half-universal'' multiplicative unitary associated to $\bG$ (the element $\widehat{\cV}$ of \cite[p.307]{kus-univ}). Then:
  \begin{itemize}
  \item we have
    \begin{equation*}
      X=(\id\otimes \pi_X)U
    \end{equation*}
    by \cite[Proposition 5.2]{kus-univ};
  \item and in general, even without the coamenability assumption,
    \begin{equation*}
      C_0(\widehat{\bG})\ni x\mapsto U(x\otimes 1)U^*\in M(C_0(\widehat{\bG})\otimes C_0^u(\widehat{\bG}))
    \end{equation*}
    is the flipped coaction of $C_0^u(\widehat{\bG})$ on its reduced version by comultiplication (as follows from the formula for $\Delta_{\widehat{\bG}}$ on \cite[p.294]{kus-univ}). 
  \end{itemize}
  It follows, then, that the rightmost element of \Cref{eq:applytheta} satisfies
  \begin{equation}\label{eq:inl2gh}
    (1\otimes\pi(a))X(x\otimes 1)X^*\in M(C_0(\widehat{\bG})\otimes K(\cH)).
  \end{equation}
  Now, since we are assuming $\widehat{\bG}$ is coamenable, $C_0(\widehat{\bG})\cong C_0^u(\widehat{\bG})$ is equipped with a counit $\varepsilon:C_0(\widehat{\bG})\to \bC$ \cite[Theorem 3.1]{bt}. It follows that we can apply that counit to the left leg of \Cref{eq:inl2gh} obtaining, in the end, a representation
  \begin{equation*}
    \theta':=(\varepsilon\otimes\id)\theta: C_0(\widehat{\bG})\ltimes_r A\to B(\cH). 
  \end{equation*}
  That that representation gives back the original \Cref{eq:pi} and \Cref{eq:X} by
  \begin{equation*}
    \pi = \theta'\circ\rho\quad\text{and}\quad X = (\id\otimes\theta')(W\otimes 1)
  \end{equation*}
  for the multiplicative unitary
  \begin{equation*}
    W\in M(C_0(\bG)\otimes C_0(\widehat{\bG}))
  \end{equation*}
  is now a simply matter of unwinding the construction, which we omit.
\end{proof}

\begin{remark}
  The proof of \Cref{th:samecrossed} given above follows the same general plan as that of \cite[Theorem 3.7]{lprs}: the latter covers the case when $\widehat{\bG}$ is classical, so that indeed $\bG$ is dual-coamenable (because classical locally compact groups are coamenable).
\end{remark}

\section{Restricting representations to finite-covolume subgroups}\label{se:resfincovol}

The aim of the section is to produce a quantum version of \cite[Proposition 2.2]{klm2}: if $\bG/\bH$ has a finite $\bG$-invariant measure then unitary $\bG$-representations whose restriction to $\bH$ is type-I are themselves type-I. Recall (e.g. \cite[Definition 5.4.2 and \S\S 5.5.1 and 13.9.4]{dixc}):

\begin{definition}\label{def:t1}
  A representation $\pi:A\to B(\cH)$ of a $C^*$-algebra is {\it type-I (or of type I)} if the commutant $\pi(A)'$ is type-I as a von Neumann algebra \cite[A 35]{dixc}.

  For an LCQG $\bG$, a unitary representation $U\in M(C_0(\bG)\otimes K(\cH))$ is {\it type-I} if the associated $C^*$-algebra representation $\pi_U:C_0^u(\widehat{\bG})\to B(\cH)$ of \Cref{def:urep} is type-I in the previous sense.

  $\bG$ itself is {\it type-I} if all of its unitary representations are.
\end{definition}

\begin{theorem}\label{th:exp}
  Fix
  \begin{itemize}
  \item a closed embedding $\iota:\bH\to \bG$ of LCQGs;
  \item a $\bG$-invariant normal state $\theta\in L^{\infty}(\bG/\bH)_*$;
  \item and a unitary $\bG$-representation $U\in M(C_0(\bG)\otimes K(\cH))$. 
  \end{itemize}
  There is a normal conditional expectation $E:R(\bH)'\to R(\bG)'$. 
\end{theorem}

Although crossed products do not feature directly here, we work our way back into the topic later, as part of the same circle of ideas. The proof strategy is very much parallel to that of \cite[Proposition 2.2]{klm2}. Let $\iota:\bH\le \bG$ be a closed quantum subgroup of a locally compact quantum group. For a unitary $\bG$-representation $U\in M(C_0(\bG)\otimes K(\cH))$ we write
\begin{equation*}
  U|_{\bH}\text{ or }U_r\in M(C_0(\bH)\otimes K(\cH))
\end{equation*}
for its {\it restriction to $\bH$}: see \cite[Proposition 6.5]{mrw} and \cite[\S 2.2]{bcv}. We also write, given such a representation, $R(\bG)$ and $R(\bH)$ for the von Neumann subalgebras  of $B(\cH)$ generated by $\bG$ and $\bH$ respectively, i.e. the weak$^*$ closures of
\begin{equation*}
  \{(\omega\otimes\id)U\ |\ \omega\in L^{\infty}(\bG)_*\}\text{ and }\{(\omega\otimes\id)U|_{\bH}\ |\ \omega\in L^{\infty}(\bH)_*\}
\end{equation*}
respectively. We naturally have
\begin{equation*}
  R(\bH)\subseteq R(\bG) \Longrightarrow R(\bG)'\subseteq R(\bH)'. 
\end{equation*}
To check the first inclusion (whence the second follows), note that $R(\bG)$ is the $W^*$-algebra generated by he image of the representation $\pi_U:C_0^u(\widehat{\bG})\to B(\cH)$ attached to $U$, similarly for $\bH$ and $U_r=U|_{\bH}$, and (\cite[\S 2.2]{bcv}) there is a factorization
\begin{equation*}
  \begin{tikzpicture}[auto,baseline=(current  bounding  box.center)]
    \path[anchor=base] 
    (0,0) node (l) {$C_0^u(\widehat{\bG})$}
    +(2,.5) node (u) {$C_0^u(\widehat{\bH})$}
    +(4,0) node (r) {$B(\cH)$}
    ;
    \draw[->] (l) to[bend left=6] node[pos=.5,auto] {$\scriptstyle \widehat{\iota}^u$} (u);
    \draw[->] (u) to[bend left=6] node[pos=.5,auto] {$\scriptstyle \pi_{U_r}$} (r);
    \draw[->] (l) to[bend right=6] node[pos=.5,auto,swap] {$\scriptstyle \pi_U$} (r);
  \end{tikzpicture}
\end{equation*}
through the morphism $\widehat{\iota}^u$ associated to $\iota:\bH\to \bG$.

Note also, for future reference:
\begin{align*}
  R(\bG)' &= \{T\in B(\cH)\ |\ U^*(1\otimes T)U=1\otimes T\}\numberthis\label{eq:rgprime}\\
  R(\bH)' &= \{T\in B(\cH)\ |\ U_r^*(1\otimes T)U_r=1\otimes T\}.\numberthis\label{eq:rhprime}\\
\end{align*}

Recall (e.g. \cite[Definition III.3.3]{tak1} and \cite[Definition IX.4.1]{tak2}):
\begin{definition}
  For an inclusion $A\subseteq B$ of $C^*$-algebras a {\it norm-1 projection} or {\it conditional expectation} $E:B\to A$ is an idempotent, norm-1 map onto $B$.

  When the inclusion is one of $W^*$-algebras we typically require that conditional expectations be {\it normal}, i.e. weak$^*$-continuous.
\end{definition}

\pf{th:exp}
\begin{th:exp}
  The construction is very much as in the proof of \cite[Theorem 1]{klm1}, adapted to the present quantum setting.

  Write
  \begin{equation*}
    V:=\mathrm{flip}(U)\in B(\cH)\otimes L^{\infty}(\bG)\text{ and }V_r:=\mathrm{flip}(U_r)\in B(\cH)\otimes L^{\infty}(\bH)
  \end{equation*}
  for the unitaries obtained from $U$ and $U_r=U|_{\bH}$ by interchanging tensorands, so that
  \begin{align*}
    (\id\otimes\Delta_{\bG})V &= V_{12}V_{13}\numberthis\label{eq:delv}\\
    (\id\otimes\iota_r)V &= V_{12}V_{r13}\numberthis\label{eq:iotav}
                           \quad\text{\cite[Lemma 2.9, equation (2.2)]{bcv}}\\
  \end{align*}
  The map
  \begin{equation}\label{eq:conjact}
    B(\cH)\ni T\longmapsto V(T\otimes 1)V^*\in B(\cH)\otimes L^{\infty}(\bG) 
  \end{equation}
  is the (right) conjugation action of $\bG$ on $B(\cH)$ attached to $U$, and the expectation $E$ will be
  \begin{equation*}
    R(\bH)'\ni T\stackrel{E}{\longmapsto} (\id\otimes\theta) V(T\otimes 1)V^*\in R(\bG)'.
  \end{equation*}
  We have to argue that
  \begin{enumerate}[(a)]
  \item\label{item:1} the definition indeed makes sense, i.e.
    \begin{equation}\label{eq:defe}
      T\in R(\bH)'\Rightarrow V(T\otimes 1)V^*\in B(\cH)\otimes L^{\infty}(\bG/\bH),
    \end{equation}
    so that $(\id\otimes\theta)$ is then applicable;
  \item\label{item:2} $E$ is a normal and has norm 1;
  \item\label{item:3} $E$ is the identity on $R(\bG)'\subseteq R(\bH)'$;
  \item\label{item:4} and its range is contained in $R(\bG)'$. 
  \end{enumerate}
  We tackle these in turn.
  
  {\bf \Cref{item:1}} To verify \Cref{eq:defe}, fix $T\in R(\bH)'$; then:
  \begin{align*}
    (\id\otimes\iota_r)V(T\otimes 1)V^* &= V_{12}V_{r13}(T\otimes 1\otimes 1)V_{r13}^*V_{12}^*
                                          \quad\text{\Cref{eq:iotav}}\\
                                        &=V_{12}(T\otimes 1\otimes 1)V_{12}^*
                                          \quad\text{because $T\in R(\bH)'$}\\
                                        &=V(T\otimes 1)V^*\otimes 1,
  \end{align*}
  which indeed means that
  \begin{equation*}
    V(T\otimes 1)V^*\in B(\cH)\otimes L^{\infty}(\bG/\bH)\subset B(\cH)\otimes L^{\infty}(\bG);
  \end{equation*}
  this ensures that $E$ is indeed well defined, taking care of \Cref{item:1}.

  {\bf \Cref{item:2}} $E$ is a composition of a von-Neumann-algebra morphism \Cref{eq:conjact} and a normal, (completely) positive map $\id\otimes\theta$ \cite[Theorem IV.5.13]{tak1}, so it too must be normal and completely positive. Since it is moreover clearly unital, its norm is $\|E(1)\|=1$ (as in \cite[Proposition 1.6.2]{arv}, for instance).

  {\bf \Cref{item:3}} It follows from \Cref{eq:rgprime} that
  \begin{equation*}
    T\in R(\bG)'\Rightarrow V(T\otimes 1)V^* = T\otimes 1,
  \end{equation*}
  and a further application of the unital $1\otimes\theta$ will produce $T$.

  {\bf \Cref{item:4}} We have to show that $V$ commutes with operators of the form $E(T)\otimes 1$. For a fixed $T\in R(\bH')$, that computation is as follows.
  \begin{align*}
    V(E(T)\otimes 1)V^* &= V((\id\otimes \theta)V(T\otimes 1)V^*\otimes 1)V^*
                          \quad\text{by definition}\\
                        &=(\id\otimes \id\otimes\theta)V_{12}V_{13}(T\otimes 1\otimes 1)V_{13}^*V_{12}^*\\
                        &=(\id\otimes \id\otimes\theta)(\id\otimes\Delta_{\bG})V(T\otimes 1)V^*
                          \quad\text{\Cref{eq:delv}}\\
                        &=(\id\otimes \theta)V(T\otimes 1)V^*\otimes 1
                          \quad\text{$\bG$-invariance of $\theta$}\\
                        &=E(T)\otimes 1.
  \end{align*}
  This concludes the proof.
\end{th:exp}

Consequently:

\begin{corollary}\label{cor:slift}
  With the same hypothesis as in \Cref{th:exp}, if $R(\bH)'$ is semifinite or of type I then so is $R(\bG)'$. 
\end{corollary}
\begin{proof}
  Indeed, by \cite[Theorems 3 and 4]{tomi3} semifinite (type-I) von Neumann algebras only admit normal expectations onto semifinite (respectively type-I) von Neumann subalgebras.
\end{proof}

The type-I branch of the statement says that $\bG$-representations are type-I provided their restrictions to $\bH$ are: this, classically, is \cite[Proposition 2.2]{klm2}. Specializing again to {\it all} unitary representations (\cite[Theorem 1]{klm1} being the classical analogue):

\begin{corollary}\label{cor:wlift}
  Let $\bH\le \bG$ be a closed embedding of LCQGs such that $L^{\infty}(\bG/\bH)$ has an invariant normal state.

  If $\bH$ is of type I then so is $\bG$.  \qedhere
\end{corollary}

\section{Invariant measures on (compact) quantum spaces}\label{se:chvol}

The ``compact quantum spaces'' in question are simply unital $C^*$-algebras. The section connects back to the preceding material as follows.

The authors of \cite{gk} address type-I-lifting (classical) results analogous to \cite[Theorem 1]{klm1} and \cite[Proposition 2.2]{klm2} by
\begin{itemize}
\item noting that for a cocompact embedding $\bH\le \bG$ of locally compact groups such that $\bG/\bH$ has a finite $\bG$-invariant measure the canonical map
\begin{equation}\label{eq:initfaith}
  C_0^u(\widehat{\bG})\to C_0^u(\widehat{\bG})\ltimes_f C(\bG/\bH)
\end{equation}
is an embedding \cite[proof of Proposition 4.2]{gk};
\item and then leveraging {\it imprimitivity} \cite[\S 3.7]{mack-unit} to recast $\bH$-representations as $C_0^u(\widehat{\bG})\ltimes_f C(\bG/\bH)$-representations (see \cite[p.275]{gk} as well as \cite[Introduction]{gtm}, where this is further elaborated upon).
\end{itemize}
Some of the results (e.g. \cite[Corollary 4.5]{gk}) turn out to be weaker than \cite[Theorem 1]{klm1} or \cite[Proposition 2.2]{klm2} because the latter only assumes a finite invariant measure (and no compactness), but the injectivity of \Cref{eq:initfaith} seems of interest on its own, and is what motivated the present quantum version thereof.

To make sense of the statement of \Cref{th:prestate}, recall (e.g. \cite[Definition 3.3 and Proposition 3.4]{dfsw} or \cite[Lemma 3.1]{dsv}):

\begin{definition}\label{def:inv}
  For a unitary representation $U\in M(C_0(\bG)\otimes K(\cH))$ of an LCQG a vector $\xi\in \cK$ is {\it $U$- or $\bG$-invariant} (or just plain {\it invariant}, all else being clear) if either of the following equivalent conditions holds
  \begin{itemize}
  \item $U(\eta\otimes \xi)=\eta\otimes \xi$ for all $\eta\in L^2(\bG)$. 
  \item For all $x\in C_0^u(\widehat{\bG})$ we have
    \begin{equation*}
      \pi_U(x)\xi = \varepsilon(x)\xi,
    \end{equation*}
    where $\pi_U$ is as in \Cref{def:urep} and $\varepsilon:C_0^u(\widehat{\bG})\to \bC$ is the {\it counit} of \cite[Proposition 6.3]{kus-univ}.
  \end{itemize}
\end{definition}

The above discussion on the injectivity of \Cref{eq:initfaith} extends to actions on {\it non}-unital $C^*$-algebras, much as in \cite[Lemma 4.1]{gtm}. First, as usual \cite[Definition II.6.2.1]{blk}, a {\it state} $\theta$ on a (possibly non-unital) $C^*$-algebra $A$ is a positive linear functional of norm $1$. The strictly-continuous extension of $\theta$ to $M(A)$ is then a state in the usual sense, i.e. unital \cite[Proposition II.6.2.5]{blk}.

For an action $\rho:A\to M(C_0(\bG)\otimes A)$, a functional $\theta\in A^*$ is {\it invariant} if
\begin{equation*}
  (\id\otimes \theta)\rho = \theta(\cdot) 1: A\to M(C_0(\bG)).
\end{equation*}

\begin{theorem}\label{th:prestate}
  Let $\rho:A\to M(C_0(\bG)\otimes A)$ be an action of an LCQG on a $C^*$-algebra $A$, and consider the following conditions.
  \begin{enumerate}[(a)]    
  \item\label{item:5} There is a $\rho$-invariant state $\phi\in A^*$.
  \item\label{item:6} There is a $\rho$-equivariant representation $(U,\pi)$ having a non-zero $\bG$-invariant vector. 
  \item\label{item:7} The canonical non-degenerate morphism
    \begin{equation}\label{eq:canemb}
      C_0^u(\widehat{\bG})\to M(C_0^u(\widehat{\bG})\ltimes_f A)
    \end{equation}
    of \Cref{eq:wantnd} is one-to-one.  
  \end{enumerate}
  We have
  \begin{equation*}
    \text{\Cref{item:5}} \Longleftrightarrow \text{\Cref{item:6}} \Longrightarrow \text{\Cref{item:7}}. 
  \end{equation*}
\end{theorem}
\begin{proof}  
  We prove the three claimed implications separately.
  
  {\bf \text{\Cref{item:5}} $\Longrightarrow$ \text{\Cref{item:6}}.} Here, it will be convenient to assume that $A$ is unital. This is always achievable:

  An action $\rho$ extends to a unital map
  \begin{equation*}
    \widetilde{A}\to M(C_0(\bG)\otimes A),
  \end{equation*}
  which in fact can be regarded as taking values in $M(C_0(\bG)\otimes \widetilde{A})$ (because $\rho$ did: \Cref{def:act} \Cref{item:13}). Naturally, the resulting morphism
  \begin{equation*}
    \widetilde{\rho}:\widetilde{A}\to M(C_0(\bG)\otimes \widetilde{A})
  \end{equation*}
  is non-degenerate (being unital), and its coassociativity is a simple check. It follows, then, that an action $\rho$ on $A$ induces another, $\widetilde{\rho}$, on the unitization $\widetilde{A}$. Invariant states or equivariant representations with invariant vectors therein also transport over from $A$ to $\widetilde{A}$ in the obvious fashion, so that the items in \Cref{item:5} or \Cref{item:6} for $\rho$ are in bijection, respectively, with the same items for $\widetilde{\rho}$.

  Throughout the proof of the current implication we will thus assume that $A$ is unital. Let $(\cK_{\phi},\ \pi_{\phi},\ \Lambda_{\phi})$ be the GNS representation \cite[\S 6.4]{blk} associated to an invariant state $\phi$, so that
  \begin{equation*}
    \pi_{\phi}:A\to B(\cK_{\phi}),\quad \Lambda_{\phi}:A\to \cK_{\phi}
  \end{equation*}
  are, respectively, a representation and a map satisfying
  \begin{equation*}
    \phi(a) = \braket{\xi_{\phi}\mid a\xi_{\phi}}
  \end{equation*}
  for the unit vector $\xi_{\phi}:=\Lambda_{\phi}(1)$ (recall that $A$ is unital). This will be the $A$-half of the desired equivariant representation. The other component, $U\in M(C_0(\bG)\otimes K(\cK_{\phi}))$, is defined by
  \begin{equation}\label{eq:uast}
    U^*(\Lambda_{\varphi}(x)\otimes \Lambda_{\phi}(a)) := (\Lambda_{\varphi}\otimes \Lambda_{\phi})(\rho(a)(x\otimes 1));
  \end{equation}
  here, $a$ ranges over $A$, $(L^2(\bG),\pi_{\varphi},\Lambda_{\varphi})$ is the GNS representation of the left Haar weight $\varphi$, and
  \begin{equation*}
    x\in \fn_{\varphi}
    :=
    \{y\in \li{\bG}\ |\ \varphi(y^*y)<\infty\}
    ,
  \end{equation*}
  i.e. $x$ is {\it square-integrable} with respect to that weight.

  This is a variant of the usual construction, employed in \cite[Proposition 3.17]{kvcast} to define the multiplicative unitary of $\bG$. Setting $\delta=1$, this also coincides with the construction in \cite[Proposition 2.4]{vs-impl} (where our $U^*$ is $V_{\theta}$).
  
  The latter result implies that this is indeed a unitary $\bG$-representation and that in fact $(U,\pi_{\phi})$ is covariant (\cite[second of the four displayed equations in the conclusion of Proposition 2.4]{vs-impl}, where $V_{\theta}$ is our $U^*$). As for invariant vectors, setting $a=1$ in \Cref{eq:uast} (so that $\Lambda_{\phi}(1)=\xi_{\phi}$) yields
  \begin{equation*}
    U^*(\Lambda_{\varphi}(x)\otimes \xi_{\phi}) = \Lambda_{\varphi}(x)\otimes \xi_{\phi},\ \forall x\in \fn_{\varphi}. 
  \end{equation*}
  Since $\Lambda_{\varphi}(\fn_{\varphi})\subseteq L^2(\bG)$ is dense, this means precisely that $\xi_{\phi}$ is $U$-invariant.
  
  {\bf \text{\Cref{item:5}} $\Longleftarrow$ \text{\Cref{item:6}}.} Consider an equivariant representation
  \begin{equation*}
    U\in M(C_0(\bG)\otimes K(\cK))\quad\text{and}\quad \pi:A\to B(\cK)
  \end{equation*}
  as in the statement, with a $U$-invariant unit vector $\xi\in \cK$. The claim is that the state
  \begin{equation}\label{eq:phidef}
    \phi(a):=\braket{\xi\mid \pi(a)\xi},\ a\in A
  \end{equation}
  is $\rho$-invariant. Indeed, for $\eta,\ \zeta\in L^2(\bG)$ we have 
  \begin{align*}
    \phi((\omega_{\eta,\zeta}\otimes\id)\rho(a)) &= \braket{\xi \mid (\omega_{\eta,\zeta}\otimes\pi)\rho(a)\xi}
                                                   \quad\text{by \Cref{eq:phidef}}\\
                                                 &= \braket{\eta\otimes\xi \mid (\id\otimes\pi)\rho(a)(\zeta\otimes\xi)}\\
                                                 &=\braket{\eta\otimes\xi \mid U^*(1\otimes \pi(a))U(\zeta\otimes\xi)}
                                                   \quad\text{by equivariance}\\
                                                 &=\braket{U(\eta\otimes\xi) \mid (1\otimes \pi(a))U(\zeta\otimes\xi)}\\
                                                 &=\braket{\eta\otimes\xi \mid (1\otimes \pi(a))(\zeta\otimes\xi)}
                                                   \quad\text{by the $U$-invariance of $\xi$}\\
                                                 &=\braket{\eta\mid\zeta}\phi(a).
  \end{align*}
  This is what $\phi$ being $\rho$-invariant means, so we are done.

  {\bf \text{\Cref{item:6}} $\Longrightarrow$ \text{\Cref{item:7}}.} We have to argue that under the hypothesis \Cref{item:6} some $\rho$-equivariant representation $(U,\pi)$ is faithful on $C_0^u(\widehat{\bG})$ when restricted along \Cref{eq:canemb}.

  We already have a covariant representation $(U,\pi)$ on say, a Hilbert space $\cK$, as in \Cref{item:6}. Next, fix a unitary representation
  \begin{equation*}
    V\in M(C_0(\bG)\otimes K(\cH))
  \end{equation*}
  so that the morphism $\pi_V:C_0^u(\widehat{\bG})\to B(\cH)$ via \cite[Proposition 5.2]{kus-univ} is faithful. The target covariant representation $(U',\pi')$ will now have carrier space $\cK\otimes \cH$, and its components are defined as follows.
  \begin{itemize}
  \item The $\bG$-representation-component $U'$ is the {\it tensor product} of $U$ and $V$ (as $\bG$-representations), denoted by `$\tpr$' in \cite[\S 1]{bcv}:
    \begin{equation}\label{eq:defu'}
      U':= U\tpr V:=V_{13}U_{12}\in M(C_0(\bG)\otimes K(\cK)\otimes K(\cH)). 
    \end{equation}
  \item As for the $A$-representation half $\pi'$, simply set
    \begin{equation}\label{eq:defpi'}
      A\ni a\stackrel{\pi'}{\longmapsto} \pi(a)\otimes 1\in B(\cK)\otimes B(\cH)\subseteq B(\cK\otimes \cH).
    \end{equation}
  \end{itemize}
  The equivariance condition is immediate: for $a\in A$ we have 
  \begin{align*}
    (\id\otimes\pi')\rho(x) &= (\id\otimes\pi)\rho(x)\otimes 1
                              \quad\text{by \Cref{eq:defpi'}}\\
                            &= U^*(1\otimes \pi(a))U\otimes 1
                              \quad\text{the equivariance of $(U,\pi)$}\\
                            &= U_{12}^*V_{13}^*(1\otimes \pi(a)\otimes 1)V_{13}U_{12}\\
                            &= U'^* (1\otimes\pi'(a)) U
                              \quad\text{\Cref{eq:defu',eq:defpi'}}
  \end{align*}
  as elements of $\li{\bG}\otimes B(\cK\otimes \cH)$.

  Nothing, so far, uses the existence of a $U$-invariant vector; the constructions, up to this stage, go through in general: we can always tensor an arbitrary unitary $\bG$-representation $V$ with an equivariant representation $(U,\pi)$ and obtain another such.

  Given, furthermore, a $U$-invariant unit vector $\xi$, we will argue that the morphism
  \begin{equation*}
    \pi_{U'}:C_0^u(\widehat{\bG})\to B(\cK\otimes \cH)
  \end{equation*}
  attached to $U'$ (by \cite[Proposition 5.2]{kus-univ} again) is an embedding. The plan is as follows:
  \begin{enumerate}[(i)]
  \item Observe that the diagram
    \begin{equation}\label{eq:piuvdiag}
      \begin{tikzpicture}[auto,baseline=(current  bounding  box.center)]
        \path[anchor=base] 
        (0,0) node (l) {$C_0^u(\widehat{\bG})$}
        +(4,.5) node (u) {$M(C_0^u(\widehat{\bG})\otimes C_0^u(\widehat{\bG}))$}
        +(8,.5) node (r) {$B(\cK)\otimes B(\cH)$}
        +(12,0) node (rrr) {$B(\cK\otimes \cH)$}
        ;
        \draw[->] (l) to[bend left=6] node[pos=.5,auto] {$\scriptstyle \Delta_{\widehat{\bG}}$} (u);
        \draw[->] (u) to[bend left=6] node[pos=.5,auto] {$\scriptstyle \pi_U\otimes \pi_V$} (r);
        \draw[->] (r) to[bend left=6] node[pos=.5,auto] {$\scriptstyle \subseteq$} (rrr);
        \draw[->] (l) to[bend right=6] node[pos=.5,auto,swap] {$\scriptstyle \pi_{U'}$} (rrr);
      \end{tikzpicture}
    \end{equation}
    commutes, i.e.
    \begin{equation}\label{eq:piuv}
      \pi_{U'} = (\pi_U\otimes \pi_V)\Delta_{\widehat{\bG}} : C_0^u(\widehat{\bG})\to B(\cK)\otimes B(\cH)\subseteq B(\cK\otimes \cH).
    \end{equation}
  \item Since $\xi\in \cK$ is $U$-invariant (\Cref{def:inv}), we have
    \begin{equation}\label{eq:xinv}
      \pi_U(x)\xi = \varepsilon(x)\xi,\ \forall x\in C_0^u(\widehat{\bG}).
    \end{equation}
  \item Whence it follows that for all $x\in C_0^u(\widehat{\bG})$ and vectors $\eta\in \cH$ we have
    \begin{align*}
      \pi_{U'}(x)(\xi\otimes\eta) &= (\pi_U\otimes \pi_V)\Delta_{\widehat{\bG}}(x) (\xi\otimes \eta)
                                    \quad\text{\Cref{eq:piuv}}\\
                                  &= \xi\otimes \pi_V((\varepsilon\otimes\id)\Delta_{\widehat{\bG}}(x))\eta
                                    \quad\text{\Cref{eq:xinv}}\\
                                  &= \xi\otimes \pi_V(x)\eta
                                    \quad\text{\cite[Proposition 6.3]{kus-univ}}.
    \end{align*}
  \end{enumerate}
  Since $\pi_V$ is assumed faithful this computation implies the faithfulness of $\pi_{U'}$, finishing the proof modulo the commutativity of \Cref{eq:piuvdiag}.

  For the latter, recall that the correspondence $U\leftrightarrow \pi_U$ is given by \Cref{eq:upiu}. The definition \Cref{eq:defu'} of $U'$ then recovers that unitary as
  \begin{align*}
    U' &= (\id\otimes\pi_U\otimes \pi_V) \wW_{13} \wW_{12}
    \\
       &= (\id\otimes\pi_U\otimes \pi_V) (\id\otimes\Delta_{\widehat{\bG}}) \wW
         \quad\text{\cite[line preceding (6.1)]{kus-univ}},
  \end{align*}
  showing that indeed \Cref{eq:piuvdiag} commutes. 
\end{proof}

To return to the compact (i.e. unital) case that provided the initial motivation as explained at the start of this section:

\begin{corollary}\label{cor:cpct}
  Let $\rho:A\to M(C_0(\bG)\otimes A)$ be an action of an LCQG on a unital $C^*$-algebra $A$. If $A$ has a $\rho$-invariant state then \Cref{eq:canemb} is in fact an embedding
  \begin{equation*}
    C_0^u(\widehat{\bG})\subseteq C_0^u(\widehat{\bG})\ltimes_f A.
  \end{equation*}
\end{corollary}
\begin{proof}
  The functoriality \cite[Proposition 4.7]{vrgn-phd} of the $C_0^u(\widehat{\bG})\ltimes_f -$ construction attaches \Cref{eq:canemb} to the equivariant non-degenerate morphism $\bC\to M(A)$, and the same result shows that when $A$ is unital, so that $\bC\to M(A)$ in fact takes values in $A$, \Cref{eq:canemb} similarly factors through
  \begin{equation*}
    C_0^u(\widehat{\bG})\cong C_0^u(\widehat{\bG})\ltimes_f \bC\to C_0^u(\widehat{\bG})\ltimes_f A.
  \end{equation*}
  \Cref{th:prestate} delivers the conclusion.
\end{proof}

\begin{remark}
  We saw in the course of the proof of \Cref{th:prestate} that we can always extend an action $\bG\circlearrowright A$ to the smallest unitization $\widetilde{A}$. By contrast, as \cite[\S 3.1]{vrgn-phd} notes, actions will not extend, in general, to {\it multiplier} algebras.
  
  This is already clear classically: in that case the multiplier algebra $M(C_0(X))$ is the function algebra $C(\beta X)$ on the {\it Stone-\v{C}ech compactification} $\beta X$ of $X$ \cite[exercise 2.C]{wo}. Now, if $\bG$ is a (classical) locally compact group, the extension of the standard translation action $\bG\circlearrowright \bG$ to $\bG\circlearrowright \beta\bG$ is continuous exactly when $\bG$ is either discrete or compact \cite[paragraph preceding Theorem 4.2]{brk}.
\end{remark}



\addcontentsline{toc}{section}{References}

\Addresses

\end{document}